\documentclass[twoside]{amsart}

\numberwithin{equation}{section}
\usepackage{amsmath}
\usepackage{amssymb,amsfonts,amsthm,latexsym,psfrag}
\usepackage{graphicx,picins,floatflt,color,epic,eepic,bbold,enumerate}
\usepackage[all]{xy}

\theoremstyle{plain}
        \newtheorem{theorem}{Theorem}[section]
        \newtheorem{lemma}[theorem]{Lemma}
        \newtheorem{definition}[theorem]{Definition}
        \newtheorem{proposition}[theorem]{Proposition}

\theoremstyle{definition}
\newcommand {\cc} [1] {\overline {{#1}}}
\newcommand {\id} {\operatorname{id}}
\newcommand {\Kern} {\operatorname{ker}}
\newcommand {\Bild} {\operatorname{im}}
\newcommand {\conv} {\operatorname{conv}}

\newcommand {\res} {\operatorname{res}}
\newcommand {\qres} {\widetilde{\operatorname{res}}}
\newcommand {\qkos} {\widetilde{\partial}}
\newcommand {\prol} {\operatorname{ext}}

\newcommand {\fnu} {[\![\nu]\!]}
\newcommand {\ii} {\operatorname{i}}
\newcommand {\ee} {\operatorname{e}}
\newcommand{\bs}{\boldsymbol}

\def\connsum{\mathop{\#}\limits}

\begin{document}
\title[On the existence of star products \dots]{On the existence of star products on quotient spaces 
of linear Hamiltonian torus actions}

\author{Hans-Christian Herbig, Srikanth~B.~Iyengar and Markus~J.~Pflaum}
\address{Fachbereich Mathematik, Ernst-Moritz-Arndt Universit\"at, Greifswald, 
         Germany}
\email{herbig@uni-greifswald.de}  
\address{Department of Mathematics, 
University of Nebraska, Lincoln NE 68588, U.S.A.}
\email{siyengar2@math.unl.edu}    
\address{Department of Mathematics, University of Colorado,
Boulder, CO 80309-0395, U.S.A.}   
\email{Markus.Pflaum@colorado.edu} 

\begin{abstract}
We discuss BFV deformation quantization \cite{BHP} in the special case of a linear Hamiltonian torus action. In particular, we show that the Koszul complex on the moment map of an effective linear Hamiltonian torus action is acyclic. We rephrase the nonpositivity condition  of \cite{AGJ} for linear Hamiltonian torus actions. It follows that reduced spaces  of such actions admit continuous star products.
\end{abstract}

\maketitle
\tableofcontents
\section{Introduction}
The purpose of this note is to elaborate on the domain of applicability of the following theorem in the special situation of linear Hamiltonian torus actions. 
In particular, we will show that in this special case condition (\ref{ci}) below is essentially always fulfilled and condition (\ref{gh}) has a simple geometric meaning. These results might be of independent interest.
\begin{theorem}[\cite{BHP}]\label{qth} Let $G$ be a compact, connected Lie group acting in a Hamiltonian fashion on a symplectic manifold $(M,\omega)$. Let $J:M\to \mathfrak g^*$ be an equivariant moment map for this action satisfying the following requirements
\begin{enumerate}[\quad\rm(1)]
\item \label{gh} for every $f\in \mathcal C^\infty(M)$ such that the restriction of $f$ to the zero fibre $Z:=J^{-1}(0)$ vanishes there exist a smooth $F: M\to \mathfrak g$ such that $f=\langle J,F\rangle$,
\item \label{ci} the Koszul complex on $J$ over the ring $\mathcal C^\infty(M)$ is acyclic. 
\end{enumerate} 
Then there exists a continuous formal deformation quantization of the (possibly singular) reduced space $M_0=Z/G$. 
\end{theorem}

Throughout the paper $\mathfrak g$ stands for the Lie algebra of the Lie group $G$. We denote the dual pairing between $\mathfrak g$ and its dual space $\mathfrak g^*$ by $\langle \:,\:\rangle$.

Let us recall what is meant by the Koszul complex on the map $J$. The space of Koszul chains is defined to be the free $\mathcal C^\infty(M)$-module $K_\bullet=K_\bullet\left(\mathcal C^\infty(M),J\right):=\wedge^\bullet\mathfrak g\otimes \mathcal C^\infty(M)$. Here the tensor product is taken over the ground field $\mathbb K$ which will be the field of real numbers $\mathbb R$ or the field of complex numbers $\mathbb C$. The Koszul differential $\partial:K_\bullet\to K_{\bullet-1}$ is given by contraction with $J$. By choosing a basis $e_1,\dots,e_\ell$ for $\mathfrak g$ with dual basis $e^1,\dots,e^\ell$ we may write $\partial=\sum_{i=1}^\ell J_i\:\iota(e^i)$. Here the components  $J_i:=J_{e_i}:=\langle J,e_i\rangle$ are given by pairing the result of the map $J$ with the basis vector $e_i$ and $\iota(e^i)$ means contraction with $e^i$. It is natural to augment the Koszul complex with the restriction map $\res:K_0=\mathcal C^\infty(M)\to\mathcal C^\infty(Z)$. By definition of $C^\infty(Z)$ the restriction map is onto. The first terms of the augmented Koszul complex are
\[0\:{\longleftarrow}\mathcal C^\infty(Z)\stackrel{\res}{\longleftarrow}K_0=\mathcal C^\infty(M)\stackrel{\partial_1}{\longleftarrow}K_1=\mathfrak g\otimes\mathcal C^\infty(M) \stackrel{\partial_2}{\longleftarrow}\cdots.\]
Condition (\ref{gh}) of the theorem is nothing but the exactness of the augmented Koszul complex in degree $0$. We will refer to it by saying that the \emph{generating hypothesis} holds. Condition (\ref{ci}) means that the complex is exact in all higher degrees. In that case $J$ is called a \emph{complete intersection}. A potential property of $J$, which is closely related to condition (\ref{gh}), is the following.
 
\begin{definition} We say that the moment map $J$ \emph{changes sign at $\zeta\in Z=J^{-1}(0)$} if for every $\xi\in \mathfrak g$ either one of the following assertions hold
\begin{itemize}
\item there exist an open neighborhood $U$ of $\zeta$ in $M$ such that ${J_\xi}_{|U}=0$,
\item for every open neighborhood $U$ of $\zeta$ in $M$ there exist $m_1,m_2\in U$ such that $J_\xi(m_1)>0$ and $J_\xi(m_2)<0$.
\end{itemize}
\end{definition}
\noindent 
Note that if $0$ is not in the image of $J$, then $J$ does not change sign by definition. It is known \cite[Propositions 6.6 and 6.7]{AGJ} that a moment map which satisfies the generating hypothesis necessarily changes sign at every point of the zero fibre. In the special case of a Hamiltonian torus action also the \emph{converse} holds \cite[Theorem 6.8]{AGJ}. We will see in Proposition \ref{equiv} that in the case of a linear torus action the sign change can be easily checked and is related to properties of the image of the moment map.  

The proof of Proposition \ref{equiv} and of Proposition \ref{toriareci}, which says that any moment map of a linear Hamiltonian torus action is a complete intersection, will occupy the section \ref{HTA} and section \ref{prop21}. Actually the arguments used are more or less elementary, the basic thread being to treat the quadratic equations by linear means. Despite their simplicity the results are important because they provide a complete and uniform picture and lots of examples of quantizable singular spaces.  The moral is that one is now in a position to attack the remaining cases, including the class of polarized torus actions \cite{GGK}, the most basic example being the harmonic oscillator. For the bulk of those torus actions the constraint surface will not be first class, and the method of homological phase space reduction (cf. section \ref{quant}) does not apply without modification. In the physics literature there exist several proposals how to rectify the situation (see, e.g., \cite{Henneaux})  and the Hamiltonian torus actions might provide a good testing ground for those ideas. We would like to stress that, in contrast to the abelian case, in the case of linear Hamiltonian actions of nonabelian groups one cannot expect such an uniform answer. This is because there are prominent examples where the reducedness question (which is related to condition (\ref{gh}) above) is notoriously difficult, see for instance \cite{Kn}. On the other hand, there are many nonabelian examples which are not complete intersections, e.g., angular momentum in dimension $\ge 3$.

We assume that the reader is acquainted with the basic notions of symplectic geometry and deformation quantization, a good reference is the book \cite{Waldmann}.  Some of the arguments also use basic commutative algebra. We provide details for the convenience of the reader.
%Since many of the people working in the field of homological physics are not to%o familiar with commutative algebra we have decided to present those arguments %in great detail.

\vspace{4mm}
\noindent\textbf{Acknowledgements.} The authors would like to thank E. Lerman and M. Davis for stimulating discussions. S.B.I. was partly supported by NSF grant DMS 0602498. 
Part of the work of S.B.I. was done at the University of Paderborn on a visit made possible by a For\-schungs\-preis from the Humboldt\-stiftung. The research of M.P. and H.-C.H. has been financially supported by the Deutsche Forschungsgemeinschaft.

\section{Linear Hamiltonian torus actions}\label{HTA}
Let $G:=\mathbb T^\ell$ be the $\ell$-dimensional torus, i.e., the $\ell$-fold  cartesian product of circles. Recall \cite{GGK} that any linear Hamiltonian $G$-action can be written as an action on the flat K\"ahler space $(\mathbb C^n,-2\omega_0)$ as follows. In order to avoid annoying prefactors we take the liberty to rescale the standard K\"ahler form $\omega_0=\ii/2\sum_i dz_i\wedge d\cc{z}_i$. Using the basis $e_1,\dots,e_\ell$ of $\mathfrak g$ we are able to record the data of the action into the matrix of weights $A=(a_{ij})\in \mathbb Z^{\ell\times n}$. Writing $g\in G$ as 
\[g=\left(\ee^{2\pi\ii\xi^1},\ee^{2\pi\ii \xi^2},\dots,\ee^{2\pi\ii \xi^\ell}\right)\]
 for $(\xi^1,\xi^2,\dots,\xi^\ell)\in\mathbb R ^\ell\cong \mathfrak g$, $g$ acts on $z=(z_1,z_2,\dots,z_n)$ according to  $z_j\mapsto \ee^{-2\pi\ii \sum_i a_{ij}\xi^i}z_j$ for $j=1,2,\dots,n$. There is a unique homogeneous quadratic moment map $J:\mathbb C^n\to \mathbb R^\ell\cong \mathfrak g^*$ for this $G$-action 
\begin{equation}\label{quadmom}
J_i(z,\cc z)=\sum_{j=1}^n a_{ij} \:z_j \cc z_j\qquad i=1,\dots \ell.
\end{equation}
A moment map for a Hamiltonian action of a torus $G$ is unique up to a constant in $\mathfrak g^*$.

Using Gaussian elimination we find that 
upon an integer change of the basis of $\mathfrak g$ and permutations of the coordinates of $\mathbb C^n$ the weight matrix $A$ can be brought into upper triangular form
\begin{equation}
\left(\begin{array}{ccccc|cc}
\diamondsuit&0 &0 &\cdots & &*&\cdots\\
0&\diamondsuit &0 &\cdots & &*&\cdots\\
0&0&\diamondsuit  &\cdots & &\cdots &\\
\cdots&&          &       & &\cdots &\\
0&0&\cdots&   & \diamondsuit &*&\cdots\\\hline
0&0&\cdots&&&0&\cdots\\
\cdots&&&&&\cdots&
\end{array}\right)=:
\left(\begin{array}{c}
\tilde{A}\\\hline
0        \\
\end{array}\right),
\end{equation}
where $\diamondsuit$ indicates nonzero and $*$ arbitrary integer entries. 
In the above representation the lower righthand block of zeros does not occur $\Leftrightarrow$ $\operatorname{Rank}_{\mathbb Q} A=\ell$ $\Leftrightarrow$ there does not exist a compact one-parameter subgroup acting trivially $\Leftrightarrow$ there does not exist a one-parameter subgroup acting trivially $\Leftrightarrow$ the $G$-action is \emph{effective}. Otherwise we may divide out the $(\ell-\operatorname{Rank}_{\mathbb Q} A)$-dimensional torus acting trivially and consider the resulting effective action of a $(\operatorname{Rank}_{\mathbb Q} A)$-dimensional torus with weight matrix $\tilde A$. 

In order to prove that the Koszul complex on the homogeneous quadratic moment map of equation (\ref{quadmom}) for $A$ of full rank $\ell$ is acyclic one can argue as follows. Using Gaussian elimination one can show that the homogeneous ideal generated by $J_1,\dots, J_\ell,z_{\ell+1}\cc z_{\ell+1},\dots,z_n\cc z_n$ in $\mathbb K[\bs z,\cc{\bs z}]$ coincides with the ideal generated by $z_1\cc z_1,\dots,z_n\cc z_n$. It is therefore of maximal height $n$. Hence, by Theorem 17.4 (iii) of \cite{Matsumura}  $J_1,\dots, J_\ell,z_{\ell+1}\cc z_{\ell+1},\dots,z_n\cc z_n$
is a regular sequence in $\mathbb K[\bs z,\cc{\bs z}]$. By the same theorem we conclude that the subsequence $J_1,\dots, J_\ell$ is also a regular in $\mathbb K[\bs z,\cc{\bs z}]$. The claim now follows easily (for more details cf. section \ref{prop21}). With a little bit more technique one can also show the acyclicity for inhomogeneous $J$, the proof will be postponed to section \ref{prop21}. 

\begin{proposition} \label{toriareci} For any effective linear Hamiltonian torus action with (not necessarily homogeneous quadratic) moment map $J$ the Koszul complex on $J$ is acyclic. 
\end{proposition}

The next result provides several characterizations of linear Hamiltonian torus actions which fulfill the generating hypothesis.

\begin{proposition} \label{equiv} For a linear Hamiltonian torus action with moment map $J=(J_1,\dots,J_\ell):\mathbb C^n\to \mathbb R^\ell$, $J_i(z,\cc z)=\sum_{j=1}^n a_{ij} \:z_j \cc z_j$ and corresponding nonzero weight matrix $A=(a_{ij})\in \mathbb Z^{\ell\times n}$ of rank $r\le\ell$ the following statements are equivalent 

\begin{enumerate}[\quad\rm(i)]
%\begin{enumerate}
%\renewcommand{\labelenumi}{\textit{(\roman{enumi})}}
\item \label{np} $J$ changes sign at every $\zeta\in Z=J^{-1}(0)$,
\item \label{vs} the image of $J$ is a vector subspace of $\mathbb R^\ell$ of dimension $r$,
\item \label{conv} $0$ is in the relative interior of the convex hull $\operatorname{conv(A)}\subset \mathbb R^\ell$ of the set of column vectors of $A$.
\end{enumerate}
More generally, $\mu$ is in the relative interior of the image of $J$ if and only if the shifted moment map $J_\mu:=J-\mu$ changes sign at every $\zeta\in J^{-1}_\mu(0)=J^{-1}(\mu)$.
\end{proposition}
\begin{proof} We use the standard euclidean scalar product $\mathbb R^\ell \times \mathbb R^\ell \to \mathbb R$, $(v,w)\mapsto v\cdot w$ in order to identify $\mathbb R^\ell$ with its dual space.

(\ref{np}) $\Rightarrow$ (\ref{vs}): Observe that the map $\mathbb C^n\to \mathbb R^n_+$, $(z_1,\dots, z_n)\mapsto(|z_1|^2,\dots, |z_n|^2)$ is onto. Let us also denote by $A:\mathbb R^n\to \mathbb R^\ell$ the linear map defined by the weight matrix. The image of $J$ is obviously the same as $A(\mathbb R_+^n)$, and we claim that the latter is $\operatorname{im}(A)$ as a consequence of (\ref{np}). We have to show that for any $v\in \operatorname{im}(A)$ there is an $x\in \mathbb R^n_+$ with $Ax=v$. Otherwise there would be, according to \cite[Farkas Lemma II]{Ziegler}, a $w\in \mathbb R^\ell$ such that $w^t A\ge 0$ and $w\cdot v<0$. This implies that $J(z,\cc z)\cdot w=\sum_{i,j}w_i\:a_{ij}|z_j|^2\ge 0$. The condition (\ref{np}) forces that to be a strict equality, i.e., $\sum_{i,j}w_i\:a_{ij}|z_j|^2=0$ for all $z$. It follows that $w$ has to be in $\operatorname{im}(A)^\perp$, which contradicts $w\cdot v<0$. 

(\ref{vs}) $\Rightarrow$ (\ref{conv}):
 First of all we show that $\Kern(A)\cap \mathbb R^n_+\ne\{0\}$. Let us assume the contrary. Because of (\ref{vs}) $A(\mathbb R^n_+)$ is a vector space. In particular, for every $v \in A(\mathbb R^n_+)$ it follows that $-v \in A(\mathbb R^n_+)$, i.e., there exist $x,x'\ge 0$ such that $Ax=v$ and $Ax'=-v$. Adding both equations we conclude that $v=0$. Therefore, we have  $A(\mathbb R^n_+)=0$, which implies that $A$ is the zero matrix, which proves the claim. Hence, we may assume that there exists a nonzero solution  $z=(z_1,\dots,z_n)\in\mathbb C^n$ of the equation $\sum_j a_{ij}|z_j|^2=0$. Therefore, zero is a convex linear combination $0=\sum_j a_{ij}\lambda_j$ with $\lambda_j=|z_j|^2/\|z\|^2$. It follows that the affine hull $\operatorname{aff}(A)$ of the column vectors of $A$ and $\Bild(A)$ conincide. Now we have to check that zero is not in any proper face of the polytope $\operatorname{conv}(A)$.  The latter would mean that there is a nonzero vector $v$ in $\operatorname{aff}(A) $ such that $v\cdot x\ge 0$ for all $x\in \operatorname{conv}(A)$. On the other hand, because of condition (\ref{vs}), every $w\in\operatorname{im}(A)$ can be written as $w=\alpha x$ for some $\alpha>0$ and $x\in\operatorname{conv}(A)$. We conclude that $v\cdot w\ge 0$ for all $w\in \operatorname{im}(A)$, which implies $A=0$.

(\ref{conv}) $\Rightarrow$ (\ref{np}): Let $v=(v_1,\dots,v_\ell)\in \mathbb R^\ell$ be some vector.
We have to check that $J(z)\cdot v=\sum_{i,j} v_i\:a_{ij} |z_j|^2$ changes sign. In fact if $v\in \Bild(A)^\perp$ this is true trivially. Let us therefore assume that $v\in\Bild(A)$. By assumption it cannot happen, that $v^tA\lambda\ge 0$ (or $\le 0$) for all  $\lambda=(\lambda_1,\dots,\lambda_n)\in [0,1]^n$ with $\sum_i \lambda_i=1$. Hence the vector $v^tA\in \mathbb R^n$ must have components with strictly positive and components with strictly negative entry. The claim  now follows easily.

%\psfrag{mu}{$\mu$}
%\psfrag{v}{$v$}
%\psfrag{A}{$A(\mathbb R^n_+)$}
%\psfrag{0}{$0$}
%\parpic[r]
%{$\begin{array}{c}
% \includegraphics[width=50mm]{mu2.eps}\\
%\hspace{1cm}
%\end{array}$}

\vspace{2mm}
In order to proof the last assertion assume that $\mu$ in in the image of $J$. $\mu$ is in a proper face of the polyhedral cone $A(\mathbb R^n_+)=J(\mathbb C^n)$ $\Leftrightarrow$ there is a vector $v\in\Bild(A)\subset \mathbb R^\ell$ such that $v\cdot (A\lambda-\mu)\ge 0$ for all $\lambda \in \mathbb R^n_+$ $\Leftrightarrow$
\begin{equation}\label{munp}
(J_\mu\cdot v)(z,\cc z)=\sum_{i,j}v_i(a_{ij}|z_j|^2-\mu_i)\ge 0
\end{equation}
for all $z\in \mathbb C^n$. This proves implication $\Leftarrow$. It remains to check that if inequality (\ref{munp}) holds in a neighborhood of $\zeta\in J^{-1}(\mu)$ then it must be true globally. If zero would be a regular value of $J_\mu\cdot v$ then the function $J_\mu\cdot v$ would change sign in every neigborhood of $\zeta$. Therefore zero  must be a singular value and the linear terms in the Taylor expansion of  $J_\mu\cdot v$ around $\zeta$ vanish. Hence we obtain $J_\mu\cdot v=\sum_{i,j}v_i a_{ij}|z_j-\zeta_j|^2\ge 0$. If that inequality holds in a neighborhood of $\zeta$ it is clearly fulfilled for all $z\in\mathbb C^n$.
\end{proof}

If we sharpen the condition (\ref{conv}) in the above proposition slightly then, due an observation of Bosio and Meersseman \cite{BM}, we obtain a smooth intersection $X_A:=Z\cap \mathbb S^{2n-1}$ of the zero fibre $Z=J^{-1}(0)$ with the unit sphere and hence $Z$ itself becomes a cone on the manifold $X_A$. In fact, one has to require that the location of $0$ in $\conv(A)$ is generic in the sense that if $0\in\conv(B)$ for some $\ell\times m$ submatrix $B$ of $A$ then we must have that $m>\ell$. Following \cite{BM} we will call matrices $A\in \mathbb R^{\ell\times n}$ with that property \emph{admissible}. 

\section{Examples}

We would like to indicate that there are plenty of nonorbifold quotients $M_0=Z/G$ among the torus actions covered by the three equivalent conditions of Proposition \ref{equiv}. 
%MARKUS FASSUNG
To this end observe
that by the slice theorem
the link of an orbifold singularity has to be isomorphic to the quotient
of a sphere by a finite group
action. Applied to our situation, this means that in case the reduced
space is an orbifold
the link of the lowest dimensional stratum has to be a rational homology
sphere.
Now that link is given by the quotient $Y_A := X_A / G$, so we just have
to test whether
this space  is not a rational homology sphere to find nonorbifold
quotient spaces among
the torus actions we consider.

%ALTE FASSUNG:
%According to \cite{Lermanetal} the symplectic reduced space of a linear Hamilto%nian action is an orbifold if and only if $Y_A:=X_A/G$ is a quotient of an (odd%-dimensional) sphere modulo a linear  action of a finite group $\Gamma$. In par%ticular, this implies that in the orbifold case $Y_A$ has necessarily to be a r%ational homology sphere.  

In order to understand the topology of $X_A$ it is convenient to make use of the following trick. Two admissible matrices $A_0,A_1$ are called \emph{homotopy equivalent} if they can be joined by a smooth curve $A_t$, $0\le t\le 1$, of admissible matrices. It follows from Ehresmann fibration theorem that $X_{A_0}$ and $X_{A_1}$ are diffeomorphic. 

For low dimensions $\ell=1,2,3$ it is a good idea to draw pictures of the polytopes $\conv(A)$. If a column vector appears $n$-fold in the matrix $A$ we will assign to the corresponding vector in the picture the  multiplicity $n$. The operation of multiplying all multiplicties with the same $m\ge 1$ can be interpreted as replacing the one-particle system corresponding to $A$ by the $m$-particle system with the ``total'' moment map corresponding to $mA=(A|A|...)$. 
\vspace{2mm}\\
\noindent $\bs{\ell=1:}$ If $A=(a_1,\dots, a_n)$ is a weight matrix of an $\mathbb S^1$-action then $A$ is admissible if and only if none of the entries is zero and $A$ has strictly positive as well as strictly negative entries. 
We conclude that $X_A$ is diffeomorphic to $\mathbb S^{2n_+-1}\times \mathbb S^{2n_--1}$, where $n_+$ and $n_-$ are the number of positive and negative entries of $A$, respectively.\psfrag{n+}{$n_+$}
\psfrag{n-}{$n_-$}
\begin{equation*}
 \includegraphics[width=2cm]{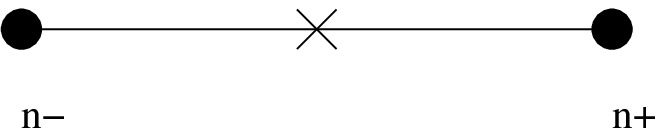}
\end{equation*}
 It is clear that an $\mathbb S^1$-quotient of $\mathbb S^{2n_+-1}\times \mathbb S^{2n_--1}$ cannot be a homology sphere for $n_+,n_-\ge 2$.
\vspace{2mm}\\
\noindent $\bs{\ell=2:}$ It is not difficult to show that any admissible matrix $A$ with $\ell=2$ rows is homotopy equivalent to a matrix corresponding to a $(2k+1)$-gon, $k\ge 1$, with  multiplicities, centered around zero 
\psfrag{n1}{$n_1$}
\psfrag{n2}{$n_2$}
\psfrag{n3}{$n_3$}
\psfrag{n4}{$n_4$}
\psfrag{n5}{$n_5$}
\begin{equation*}
\raisebox{-10mm}{\includegraphics[width=2.1cm]{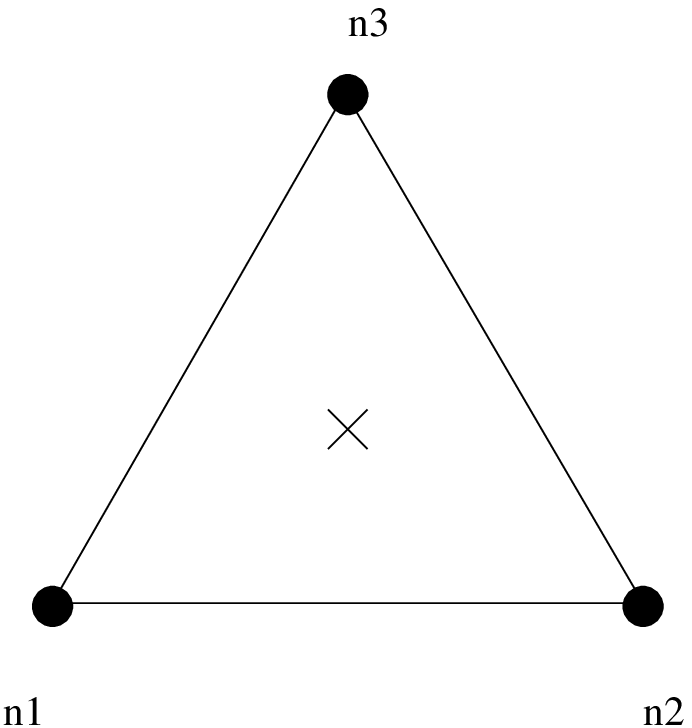}}\qquad,\qquad\raisebox{-15mm}[15mm][18mm]{\includegraphics[width=3.5cm]{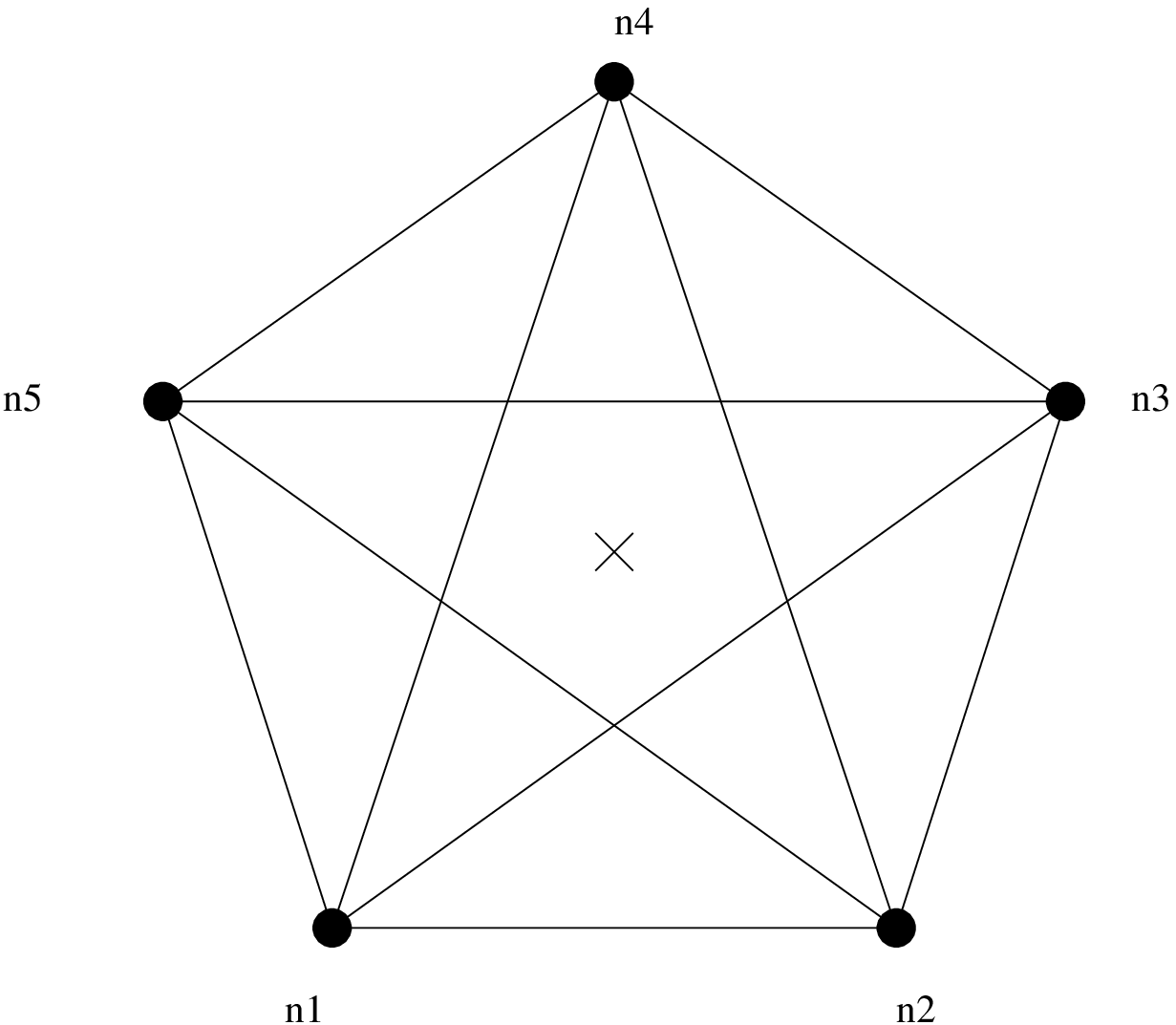}}\qquad\mbox{ etc.}
\end{equation*}
A theorem of Lopez de Medrano \cite{LdM} says that if $A$ is such an admissible $2\times n$-matrix then $X_A$ is diffeomorphic to
\begin{itemize}
\item $\mathbb S^{2n_1-1}\times \mathbb S^{2n_2-1}\times \mathbb S^{2n_3-1}$ for $k=1$, i.e., for a triangle,
\vspace{1mm}
\item $\connsum_{i=1}^{2k+1} (\mathbb S^{2d_i-1}\times \mathbb S^{2n-2d_i-2})$ for $k\ge 2$, where $\connsum$ denotes the connected sum and $d_i=n_i+\dots+n_{i+k-1}$ taken modulo $2k+1$. 
\end{itemize}
We conjecture that for fixed $k$ for all, except possibly finitely many, of those $X_A$ the quotient $X_A/\mathbb T^2$ cannot be a rational homology sphere for dimensional reasons.
\vspace{2mm}\\
\noindent\textbf{Cross-polytope $\bs{\mathbb T^\ell}$-action:} The most simple torus action with  a nonadmissible weight matrix which fullfills the requirements of Proposition \ref{equiv} is the one corresponding to the $2$-dimensional cross-polytope, that is, the one with weight matrix
\psfrag{1}{$1$}
\begin{equation*}
 A_2=\left(\begin{array}{cccc}
1&-1&0&0\\
0&0&1&-1
\end{array}\right)\qquad\mbox{and diagram }\qquad\quad\raisebox{-15mm}{\includegraphics[width=2.8cm]{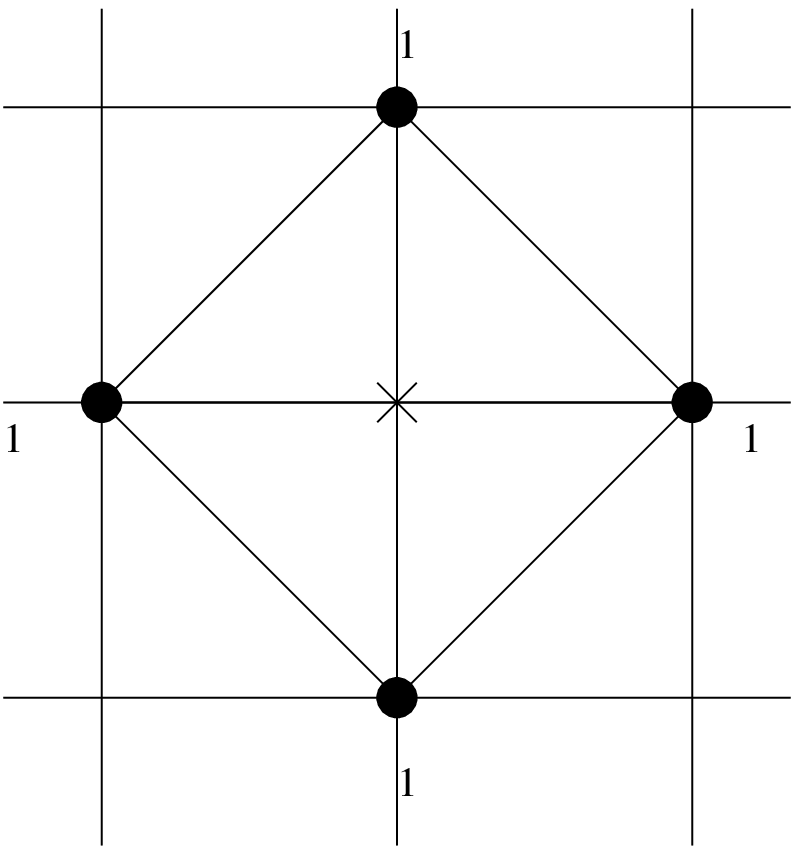}}\quad,
\end{equation*}
defining a $G=\mathbb T^2$ action on $\mathbb C^4$. The equations defining $X_{A_2}$ are $|z_1|^2=|z_2|^2=\lambda$, $|z_3|^2=|z_4|^2=1-\lambda$, $\lambda\in [0,1]$. The boundary strata at $\lambda=0,1$ are diffeomorphic to $\mathbb T^2$ and are the orbit type strata of the subgroups $\mathbb S^1\times{\id}$ and $\id\times \mathbb S^1$, respectively. An element $(u_1,u_2)=(\ee^{2\pi\ii\xi^1},\ee^{2\pi\ii \xi^2})\in \mathbb T^2$ acts on $\mathbb C^4$ by 
\[(z_1,z_2,z_3,z_4)\mapsto(u_1z_1,u_1^{-1}z_2,u_2 z_3,u_2^{-1}z_4).\]
The strata of $Y_{A_2}=X_{A_2}/\mathbb T^2$ are therefore the top stratum $\mathbb S^1\times \mathbb S^1\times (0,1)$ and the two circles which are the orbit type strata of the subgroups $\mathbb S^1\times{\id}$ and $\id\times \mathbb S^1$, respectively. It is easy to see that  $Y_{A_2}$ is homeomorphic to a $3$-sphere with an embedded Hopf link. It is also straightforward to show that  $Y_{A_2}$ is homeomorphic to the following orbit space of a finite group: take the canonical action of $\mathbb Z_{j_1}\times \mathbb Z_{j_2}$ on $\mathbb S^3\subset \mathbb C\times \mathbb C$ for any  numbers $j_1,j_2\ge 2$. The quotient $Y_{mA_2}$ of the corresponding $m$-particle system is given by replacing the $\mathbb S^1$ in the construction of $Y_{A_2}$ by $\mathbb S^{2m-1}\times \mathbb S^{2m-1}/\mathbb S^1$. Here $\mathbb S^1$ acts on $\mathbb S^{2m-1}\times \mathbb S^{2m-1}\subset \mathbb C^m\times \mathbb C^m$ by $u.(\bs z_1,\bs z_2)=(u\bs z_1,u^{-1}\bs z_2)$. We conjecture that $Y_{mA_2}$ is not a rational homology sphere for $m\ge 2$.

The discussion generalizes almost verbatim to the situation where the weight matrix is $A_\ell$, $\ell\ge 2$, the matrix whose column vectors are the vertices $\{\pm e_i\}_{i=1,\dots,\ell}$ of the $\ell$-dimensional cross-polytope. In this manner we define an action of $\mathbb T^\ell$ on $\mathbb C^{2\ell}$. In the argument we have to replace the intervall by the $(\ell-1)$-dimensional simplex $\Delta^{\ell-1}$. The poset  underlying  the orbit type stratification of $Y_{A_\ell}$ corresponds to the  poset of faces of  $\Delta^{\ell-1}$. Again, we may write $Y_{A_\ell}$ as a quotient of $\mathbb S^{2\ell-1}\subset \mathbb C^\ell$ modulo $\mathbb Z_{j_1}\times\dots\times\mathbb Z_{j_{\ell}}$ for $j_1,\dots,j_\ell\ge 2$.
%\begin{itemize}
%\item
%collection $X_n=(S^1)^{n+1}$, $n\ge 0$, regarded as a semi-simplicial topologic%al space with face maps
%$d_i:X_n\to X_{n-1}$,  $d_i(z_0,\dots,z_n):=(z_0,\dots,z_{i-1},z_{i+1},\dots,z_%n)$, $0\le i\le n$
%\item collection of standard simplices $\Delta^n$, $n\ge 0$, regarded as a cose%mi-simplicial topological space with face maps 
%$\partial^i:\Delta^n\to \Delta^{n+1}$, 
%$\partial^i(\lambda_0,\dots,\lambda_{n}):=(\lambda_0,\dots,\lambda_{i-1},0,\lam%bda_{i},\dots,\lambda_{n})$, $0\le i\le n$
%\end{itemize}
%Are you able to compute $H_\bullet (Y_\ell,\mathbb Q)$ of $Y_\ell:=\bigsqcup_{n%=0}^\ell X_n\times \Delta^n/\sim$? Here the relation $\sim$ is generated by $(\%varphi _*z,\lambda)\sim (z,\varphi^* \lambda)$ for any $z\in X_n$ $\lambda \in %\Delta^m$ and any strictly increasing $\varphi:[m]\to[n]$. $Y_0$ and $Y_1$ seem% to be rational homology spheres. Is it true in general that $Y_\ell$ is a rati%onal homology sphere of dimension $2\ell-1$? What if we replace $S^1$ by some o%ther topological space, e.g., $S^3\times_{S^1} S^3$?

\section{Proof of Proposition \ref{toriareci}}\label{prop21}

In this section we recall, with proofs, some elementary results in commutative algebra. 

A homomorphism $\varphi\colon R\to S$ of commutative rings is said to be \emph{flat} if $S$ is flat when viewed as an $R$-module via $\varphi$. 

\begin{lemma}
\label{lem:basechange}
Let $K$ be a commutative ring. Let $\varphi\colon R\to S$ and $\varphi'\colon R'\to S'$ be homomorphisms of $K$-algebras. When $\varphi$ and $\varphi'$ are flat, so is the induced homomorphism $\varphi\otimes_{K}\varphi'\colon R\otimes_{K}R'\to S\otimes_{K}S'$.
\end{lemma}

\begin{proof}
For any module $W$ over $R\otimes_{K}R'$ the natural map
\[
S\otimes_{R}W\to  (S\otimes_{K}R')\otimes_{(R\otimes_{K}R')} W,\quad s\otimes w\mapsto s\otimes 1 \otimes w
\]
is an isomorphism with inverse given by $s\otimes r'\otimes w\mapsto s\otimes r 'w$. 
Since $S$ is flat over $R$, one thus deduces that the homomorphism $\varphi\otimes_{K}R'$ is flat.

In the same vein, the homomorphism $S\otimes_{K}\varphi'$ is flat. It remains to note that 
$\varphi\otimes_{K}\varphi'$ is the composition $(S\otimes_{K}\varphi')\circ (\varphi\otimes_{K}R')$.
\end{proof}

Let $R$ be a commutative ring. We say that a sequence $\bs{r}=r_{1},\dots,r_{n}$ in $R$ is \emph{regular} if $r_{i+1}$ is a non-zero divisor on the ring $R/R(x_{1},\dots,x_{i})$ for  $0\leq i\leq n-1$.
In the literature, a regular sequence is also required to satisfy $R/R\bs r\ne 0$, but this condition is not important here.

We say that a homomorphism $\varphi\colon R\to S$ is \emph{faithful} if for any non-zero $R$-module $M$, the $S$-module $S\otimes_{R}M$ is non-zero.

\begin{lemma}\label{imageofrs}
\label{lem:flat}
Let $\bs{r}=r_{1},\dots,r_{n}$ a sequence of elements in a commutative ring $R$ and $\varphi\colon R\to S$ a homomorphism of rings that is flat. If the sequence $\bs r$ is regular in $R$, then the sequence $\varphi(\bs r)$ in $S$ is regular; the converse holds if $\varphi$ is also faithful.
\end{lemma}

\begin{proof}
For any element $r$ in $R$, the induced homomorphism 
\[
\varphi\otimes_{R}(R/Rr)\colon R/Rr\to S/S\varphi(r)
\]
is flat, and it is faithful when $\varphi$ is faithful. These claims follows from the observation that the functor $-\otimes_{R/Rr}(S/S\varphi(r))$ of $R/Rr$-modules coincides with $-\otimes_{R}S$. 

It thus suffices to verify the desired result for a single element $r$ in $R$. Note that $r$ is non-zero divisor if and only if  $\mathrm{Ker}(R\xrightarrow{\,r\,}R)=0$ holds. Since $\varphi$ is flat, there is an isomorphism
 \[
S\otimes_{R}\mathrm{Ker}(R\xrightarrow{\ r\ }R)\cong \mathrm{Ker}(S\xrightarrow{\ \varphi(r)\ }S)\,.
\]
It follows that when $r$ is a non-zero divisor, the element $\varphi(r)$ in $S$ is a non-zero divisor, and that the converse holds when $\varphi$ is faithful.
\end{proof}

\begin{proposition}\label{Jisrs}
Let $\bs J = J_{1},\dots,J_{\ell}$ be elements in the polynomial ring $K[\bs z,\cc{\bs z}]$, where each $J_{i}=\Sigma_{j=1}^{n}a_{ij}z_{j}\cc{z}_{j} + c_{i}$, with $a_{ij},c_{i}$ in $K$, for all $i,j$. If the matrix $(a_{ij})$ has rank $\ell$, then the sequence $\bs J$ is regular.
\end{proposition}

\begin{proof}
Let $K[\bs x]$ be a polynomial ring on indeterminates $\bs x =x_{1},\dots,x_{\ell}$. Consider the homomorphism of $K$-algebras 
\[
\varphi\colon K[\bs x]\to K[\bs z,\cc{\bs z}]
\]
defined by $\varphi(x_{i})=z_i\cc{z}_i$ for $i=1,\dots,\ell$.  We claim that this map is flat. 

Indeed, observe that $\varphi=\varphi_{1}\otimes_{K}\cdots\otimes_{K}\varphi_{n}$, where $\varphi_{i}\colon K[x_{i}]\to K[z_{i},\cc{z}_{i}]$ is the homomorphism of $K$-algebras defined by $\varphi(x_{i})=z_{i}\cc{z}_{i}$. In view of Lemma~\ref{lem:basechange}, it thus suffices to prove that each $\varphi_{i}$ is flat. Note that $K[z_{i},\cc{z}_{i}]$ is even free as a $K[x_{i}]$-module, on a basis $\{1\}\sqcup \{z^{e}_{i},(\cc{z}_{i})^{e}\}_{e\geqslant 1}$.

For each $1\leq i\leq \ell$, let $L_{i}$ denote the linear form $\Sigma_{j=1}^{n}a_{ij}x_{j}+c_{i}$ in $K[\bs x]$. Since $\varphi(L_{i})=J_{i}$ and $\varphi$ is flat, for the desired result it suffices to prove that the sequence $L_{1},\dots,L_{\ell}$ in $K[\bs x]$ is regular, by Lemma~\ref{lem:flat}.

Now let $\sigma\colon K[\bs x]\to K[\bs x]$ be the homomorphism of $K$-algebras with  $\sigma(x_{i})=L_{i}$ for each $i$. The rank condition on the matrix $(a_{ij})$ implies that $\sigma$ admits an inverse, and hence it is an automorphism. In particular, the map $\sigma$ is flat, so it suffices to verify that the sequence $x_{1},\dots,x_{\ell}$ is regular on $K[\bs x]$, again by Lemma~\ref{lem:flat}. The desired result is now evident
\end{proof}
\begin{proof}[Proof of  Proposition~\emph{\ref{toriareci}}] Let us introduce some notation. We denote by $\mathcal F_{\bs a}=\mathbb K[\![\bs z-\bs a,\cc{\bs z}-\cc{\bs a}]\!]$ the ring of formal power series  at $\bs{a}=(a_1,\dots,a_n)\in\mathbb C^n$ over the field $\mathbb K$ (which is $\mathbb R$ or $\mathbb C$). Similarly, $\mathcal O_{\bs a}=\mathbb K\{\bs z-\bs a,\cc{\bs z}-\cc{\bs a}\}$ stands for the ring of convergent power series at $\bs{a}\in\mathbb C^n$. The ring of germs at $\bs{a}\in\mathbb C^n$ of smooth functions on $\mathbb C^n$ will be denoted by $\mathcal E_{\bs a}$. Observe that we have the following chain of inclusions of $\mathbb K$-algebras
\begin{equation}\label{arrows}
\xymatrix{
\mathbb K[\bs z,\cc{\bs z}]\ar[r]^{\hspace{2mm}\footnotesize (1) \normalsize}_{\hspace{2mm}\mbox{\footnotesize flat}}& \mathcal F_{\bs a}&&\mathcal O_{\bs a}\ar[ll]_{\hspace{-0mm}\footnotesize (2) }^{\hspace{-0mm}\mbox{\footnotesize faithfully flat }}\ar[rr]^{\hspace{-0mm}\mbox{\footnotesize (3) }}_{\hspace{-0mm}\mbox{\footnotesize faithfully flat }}&&\mathcal E_{\bs a}.}
\end{equation}
A proof of the fact that arrow (1) is flat can be found, e.g., in \cite[Chapter 3, Exercise 7.4]{Matsumura}. For a proof of the fact that arrows (2) and (3) are faithfully flat we refer to \cite[Theorem III.4.9, Corollary VI.1.12.]{Malgrange}. From Lemma \ref{imageofrs} and Proposition \ref{Jisrs} it follows that the image of $J_1,\dots,J_\ell$ in $\mathcal E_{\bs a}$ is a regular sequence for every $\bs a$. In particular, the Koszul complex $K_\bullet(\mathcal E_{\bs a},J)$ is acyclic for every $\bs a$. Using a partition of unity argument it follows that $K_\bullet(\mathcal C^\infty(\mathbb C^n),J)$ is acyclic.
  \end{proof}
\section{Quantization}\label{quant}
For sake of completeness let us briefly describe how to find the deformation quantization of Theorem \ref{qth} of the reduced space. For proofs and an explanation of the BFV-machinery used for it we refer to \cite{BHP,HPhD}. First of all, recall that according to \cite{ACG,SL} the algebra of smooth functions on the possibly singular reduced space is the Poisson algebra $\mathcal C^\infty(M)^G/   C^\infty(M)^G\cap I_Z$. Here we denote by  $I_Z$  the ideal of smooth functions on $M$ vanishing on $Z$ and we observe that $\mathcal C^\infty(Z)=\mathcal C^\infty(M)/I_Z$. Since $G$ is compact and connected we have an isomorphism 
\[\mathcal C^\infty(Z)^{\mathfrak g}\to \mathcal C^\infty(M)^G/   C^\infty(M)^G\cap I_Z\]
of Fr\'echet  algebras which is given by extending a $G$-invariant function on $Z$ to a $G$ invariant function on $M$ and taking the representative of the latter. In a similar manner, there is an isomorhism of Fr\'echet algebras from $\mathcal C^\infty(Z)^{\mathfrak g}$ to $N(I_Z)/I_Z$, where $N(I_Z)=\{f\in\mathcal C^\infty(M)\:|\:\{f,I_Z\}\subset I_Z\}$ is the normalizer of $I_Z$ in $\mathcal C^\infty(M)$. Now the generating hypothesis means that $I_Z$ is just the ideal $\langle J_1,\dots,J_\ell\rangle_{\mathcal C^\infty(M)}$ generated by the components of the moment map. A crucial observation is that this implies that $Z$ is \emph{first class}, i.e. $I_Z$ is a Poisson subalgebra of $\mathcal C^\infty(M)$, and hence both aforementioned isomorphisms are isomorphisms of Poisson algebras. We are looking for a reduced star product $*_0$, that is, a continuous associative deformation of  this Poisson algebra in the formal parameter $\nu$.

The first ingredient for the reduced star product is a continuous splitting of the augmented Koszul complex  
\begin{equation*}
\xymatrix{0& \mathcal C^\infty(Z)\ar[l]\ar@/_1pc/[r]_{\prol}&\mathcal C^\infty(M)\ar[l]_{\res}\ar@/_1pc/[r]_{h_0}&\mathfrak g\otimes \mathcal C^\infty(M)\ar[l]_{\partial_1}\ar@/_1pc/[r]&\cdots\ar[l].}
\end{equation*}
The existence of a continuous extension map $\prol$ follows from the fact that real analytic sets have the extension property \cite{BS,BMil}. Note that if $Z$ has  singularities one cannot expect to find an extension map which is an algebra morphism. The existence of continuous contracting homotopies $h_i:K_i\to K_{i+1}$ such that $\id_{K_0}=\prol\res+\partial _1h_0$ and $\id_{K_i}=h_{i-1}\partial_i+\partial_{i+1}h_i$ for all $i\ge 1$ can be proved using the division theorem of \cite{BS}.
In addition, we have to require that $\prol$ and $h_0$ are $G$-equivariant and $h_0\prol=0$. It is not difficult to modify existing splits in order to realize these extra  conditions. The authors do not know any example of a singular zero locus $Z$ where one has explicit formulas for the split. 

The second ingredient is a strongly invariant star  product $*$ on $M$. This means by definition that 
\[J_\xi*f-f*J_\xi=\nu\{J_\xi,f\}\qquad \forall f\in \mathcal C^\infty(M),\xi\in \mathfrak g.\]
It is well-known that such a star product exists for any Hamiltonian action of a compact group on a symplectic manifold. In the case of a linear Hamiltonian action on $\mathbb C^n$ the Wick star product is obviously strongly invariant.

We now introduce a deformed version $\qkos:K_{\bullet}\fnu\to K_{\bullet-1}\fnu$ of the Koszul differential. For sake of simplicity we give here merely the formula for torus actions
\[\qkos:=\sum_i R_{J_i}\iota(e^i),\]
where we denote by $R_f:\mathcal C^\infty(M)\fnu\to\mathcal C^\infty(M)\fnu$, $h\mapsto h*f$ the operator of right multiplication with the function $f$.
It is easy to see that $\qkos$ is a differential such that  $\qkos-\partial$ is of order $\nu$. Hence the geometric series in 
\[\qres:=\res(\id-(\qkos_1-\partial_1)h_0)^{-1}\]
converges $\nu$-adically and the so-called deformed restriction map 
\[\qres=\res +\sum_{i\ge 1}\nu^i\res_i:\mathcal C^\infty(M)\fnu\to \mathcal C^\infty(Z)\fnu\] 
is in fact a formal power series of linear continuous maps $\res_i:\mathcal C^\infty(M)\to \mathcal C^\infty(Z)$.

We are now ready to write down the formula for the reduced star product $*_0$. For invariant smooth functions  $f$ and $g$ on $Z$ we define
\begin{equation}\label{redstar}
f*_0 g:=\qres\big(\prol(f)*\prol(g)\big).
\end{equation}
One still has to check that $f*_0 g$ is a formal power series of invariant functions  and that the resulting operation makes $\mathcal C^\infty(Z)^\mathfrak g\fnu$ indeed into an \emph{associative} $\mathbb K\fnu$-algebra. For a proof of these statements, which is based on the formalism of Batalin-Fradkin-Vilkovisky-quantization \cite{Henneaux} we refer to \cite{BHP,HPhD}. Even though this is not visible in formula (\ref{redstar}) substantial use is made of the acyclicity of the Koszul complex. It is an open problem whether one can get rid of the assumption that $J$ is a complete intersection.
\bibliographystyle{amsplain}
\bibliography{HIP}
\end{document}